\documentclass{amsart}

\usepackage[latin1]{inputenc}
\usepackage[T1]{fontenc}
\usepackage{enumerate}
\usepackage{mathtools}
\usepackage{amssymb}
\usepackage{mathrsfs}
\usepackage{mathabx}

\newtheorem{theorem}{Theorem}[section]
\newtheorem{lemma}[theorem]{Lemma}

\newtheorem{prop}[theorem]{Proposition}

\theoremstyle{definition}
\newtheorem{definition}[theorem]{Definition}

\theoremstyle{remark}

\numberwithin{equation}{section}

\newcommand{\N}{\mathbb N}

\title{Inverse of $\mathcal{U}$-frequently hypercyclic operators}
\author[Q. Menet]{Quentin Menet}

\address{Quentin Menet, Univ. Artois, EA 2462, Laboratoire de Mathématiques de Lens (LML), F-62300 Lens, France}
\email{quentin.menet@univ-artois.fr}
\thanks{The authors 
were supported by the grant ANR-17-CE40-0021 of the French National 
Research Agency ANR (project Front)}

\allowdisplaybreaks[3]
\begin{document}
\begin{abstract}
We show that there exists an invertible $\mathcal{U}$-frequently hypercyclic operator on $\ell^p(\N)$ ($1\le p <\infty$) whose inverse is not $\mathcal{U}$-frequently hypercyclic.
\end{abstract}
\maketitle
\section{Introduction}
Let $X$ be a separable infinite-dimensional Fréchet space and $T$ an operator on $X$. We say that $T$ is hypercyclic if there exists a vector $x$ in $X$ such that $\text{Orb}(x,T):=\{T^nx:n\ge 0\}$ is dense in $X$, or equivalently, such that for every non-empty open set $U$, the set $\mathcal{N}_T(x,U):=\{n\ge 0:T^{n}x\in U\}$ is infinite. Birkhoff~\cite{Bir} showed that $T$ is hypercyclic if and only if $T$ is topologically transitive, \emph{i.e.} for every non-empty open sets $U, V$, there exists $n\ge 0$ such that $T^nU\cap V\ne \emptyset$. In particular, if $T$ is invertible, it follows that $T$ is hypercyclic if and only if $T^{-1}$ is hypercyclic. 

Several variants of hypercyclicity have been deeply investigated during last years. For instance, Bayart and Grivaux~\cite{2Bayart0, 2Bayart} introduced the notion of frequent hypercyclicity in 2004 and the notion of $\mathcal{U}$-frequent hypercyclicity was introduced by Shkarin~\cite{Shkarin} in 2009. An operator $T$ is said to be frequently hypercyclic if there exists a vector $x$ in $X$ such that for every non-empty open set $U$, $\underline{\text{dens}}\ \mathcal{N}_T(x,U)>0$, and an operator $T$ is said to be $\mathcal{U}$-frequently hypercyclic if there exists a vector $x$ in $X$ such that for every non-empty open set $U$, $\overline{\text{dens}}\ \mathcal{N}_T(x,U)>0$. No Birkhoff-type characterization is known for frequent hypercyclicity while for $\mathcal{U}$-frequent hypercyclicity, a Birkhoff-type characterization was given by Bonilla and Grosse-Erdmann~\cite{Bon}.

An important open question posed in \cite{2Bayart} and which can also be found in \cite{BayartR, Guirao} consists in determining if the inverse of a frequently hypercyclic operator is still frequently hypercyclic. This question is also open for $\mathcal{U}$-frequently hypercyclic operators and was posed explicitly by Grosse-Erdmann~\cite{Grosse}. In fact, we know thanks to Bayart and Ruzsa~\cite{BayartR} that if $T$ is invertible and frequently hypercyclic then $T^{-1}$ is $\mathcal{U}$-frequently hypercyclic. However, it is not clear if the inverse of a frequently hypercyclic operator is frequently hypercyclic or if the inverse of a $\mathcal{U}$-frequently hypercyclic operator is $\mathcal{U}$-frequently hypercyclic. Although we have a Birkhoff-type characterization for $\mathcal{U}$-frequent hypercyclicity, it is not obvious to determine if the topological characterization of $\mathcal{U}$-frequent hypercyclicity passes to its inverse. Note that for the reiterative hypercyclicity introduced in \cite{Bes} and relying on the upper Banach density, Bonilla and Grosse-Erdmann~\cite{Bon} have shown that the inverse of a reiteratively hypercyclic operator is reiteratively hypercyclic.

We focus in this paper on the case of $\mathcal{U}$-frequently hypercyclic operators and show that there exists an invertible $\mathcal{U}$-frequently hypercyclic operator whose inverse is not $\mathcal{U}$-frequently hypercyclic. Since such a counterexample cannot be provided  by bilateral weighted shifts on $\ell^p(\mathbb{Z})$ \cite{BayartR} or on $c_0(\mathbb{Z})$ \cite{Grosse}, we will consider operators of C-type which have been introduced in \cite{Menet} in order to exhibit a chaotic operator that is not $\mathcal{U}$-frequently hypercyclic and which have been deeply investigated in \cite{Monster}. In this last paper, these operators have allowed to exhibit, among others, frequently hypercyclic operators which are not ergodic or $\mathcal{U}$-frequently hypercyclic operators which are not frequently hypercyclic on Hilbert spaces. However, each of these counterexamples were not invertible and it will be necessary to adapt several results relating on operators of C-type in order to obtain the desired counterexample.

We refer the reader to the recent books~\cite{Bayartbook, KGEbook} for more information on linear dynamics.

\section{Invertible operators of C-type} 

An operator of C-type is associated to four parameters $v$, $w$, $\varphi $, and $b$, where

\begin{enumerate}
 \item[-] $v=(v_{n})_{n\ge 1}$ is a sequence of non-zero complex numbers such 
that we have ${\sum_{n\ge 1}|v_{n}|<\infty}$;
\item[-] $w=(w_{k})_{k\geq 1}$ is a sequence of complex numbers which is both bounded  
and bounded below, \mbox{\it i.e.} $0<\inf_{k\ge 1} \vert w_k\vert\leq \sup_{k\ge 1}\vert w_k\vert<\infty$, and such that $\inf_{n\geq 0} |W_n| >0$ where $W_n=\prod_{b_n<j<b_{n+1}}  w_j$,
\item[-] $\varphi $ is a map from $\N$ into itself, such that $\varphi 
(0)=0$, $\varphi (n)<n$ for every $n\ge 1$, and the set 
$\varphi ^{-1}(l)=\{n\ge 0\,:\,\varphi (n)=l\}$ is infinite for every 
$l\ge 0$;
\item[-] $b=(b_{n})_{n\ge 0}$ is a strictly increasing sequence of positive 
integers such that $b_{0}=0$ and $b_{n+1}-b_{n}$ is a multiple of 
$2(b_{\varphi (n)+1}-b_{\varphi (n)})$ for every $n\ge 1$.
\end{enumerate}

\begin{definition}\label{Definition 43}
 The \emph{operator of C-type} $T_{v,w,\varphi,b}$ on $\ell^p(\mathbb{N})$ is defined by
\[
T_{v,w,\varphi,b}\ e_k=
\begin{cases}
 w_{k+1}\, e_{k+1} & \textrm{if}\ k\in [b_{n},b_{n+1}-1),\; n\geq 0,\\
v_{n}\, e_{b_{\varphi(n)}}-W_n^{ -1 } e_{
b_{n}} & \textrm{if}\ k=b_{n+1}-1,\ n\ge 1,\\
 -W_0^{-1}e_0& \textrm{if}\ 
k=b_1-1.
\end{cases}
\]
\end{definition}

An important characteristic of operators of C-type is that every finite sequence is periodic for these operators. Indeed, if $T$ is an operator of C-type associated to $(b_n)$ then for every $k\in [b_{n},b_{n+1})$, we can compute that $T^{2(b_{n+1}-b_n)}e_k=e_k$ (see \cite[Lemma 6.4]{Monster}).\\

We start by stating sufficient conditions on parameters $v$ and $w$ so that an operator of C-type associated to $v$ and $w$ is invertible.

\begin{prop}\label{invertible}
If $\sup_{n\geq 0} |W_n|<\infty$ and $|v_n|<\frac{1}{2^{n}\sup_{n\geq 0} |W_n|}$ for every $n\ge 1$, then the operator of C-type $T_{v,w,\varphi,b}$ is invertible on $\ell^p(\N)$ ($1\le p<\infty$) and 
\begin{itemize}
\item for every $n\geq 0$, every $k\in (b_{n},b_{n+1})$,
\[T^{-1}_{v,w,\varphi,b}\ e_k= \frac{1}{w_{k}}\, e_{k-1};\]
\item $T^{-1}_{v,w,\varphi,b}e_0= -W_0\ e_{b_1-1}$;
\item for every $n\ge 1$,
\[
T^{-1}_{v,w,\varphi,b}\ e_{b_n}=-\sum_{m=0}^{m_n-1} \Big(\prod_{l=0}^mv_{\varphi^l(n)}\Big)\,\Big(\prod_{l=0}^{m+1}W_{\varphi^{l}(n)}\Big) e_{b_{\varphi^{m+1}(n)+1}-1}-W_n \ e_{
b_{n+1}-1}\]
where $m_n=\min\{m\ge 0:\varphi^m(n)=0\}$.
\end{itemize}
\end{prop}
\begin{proof}
We first prove that $T_{v,w,\varphi,b}$ is injective. Let $x\in \ell^p(\N)$. Assume that $T_{v,w,\varphi,b}\ x=0$.
For every $n\ge 0$, for every $k\in [b_{n},b_{n+1}-1)$, we then have $w_{k+1} x_k=0$ and thus $x_k=0$. On the other hand, for every $n\ge 0$, we have \[-W_n^{ -1 }x_{b_{n+1}-1}+\sum_{m\ge 1:\varphi(m)=n}v_m x_{b_{m+1}-1}=0.\]

Let $C=\sup_n |W_n|$. Assume that  there exists $n_0$ such that $|x_{b_{n_0+1}-1}|=\varepsilon>0$. Then we deduce that there exists $n_1>n_0$ such that $\varphi(n_1)=n_0$ and 
\[|x_{b_{n_1+1}-1}|\ge \frac{2^{n_0}}{|v_{n_1}|2^{n_1}}|W_{n_0}|^{-1}|x_{b_{n_0+1}-1}|\ge \frac{C^{-1}\varepsilon 2^{n_0}}{|v_{n_1}|2^{n_1}}.\] By repeating this argument, we get an increasing sequence $(n_k)$ such that \[|x_{b_{n_k+1}-1}|\ge \frac{C^{-k}\varepsilon 2^{n_0}}{\left(\prod_{j=1}^k|v_{n_j}|\right)2^{n_k}}>\frac{C^{-k}\varepsilon 2^{n_0}}{\left(\prod_{j=1}^k\frac{1}{C2^{n_j}}\right)2^{n_k}}\xrightarrow[k\to \infty]{} \infty.\]
This is then impossible than $x$  belongs to $\ell^p(\N)$ and we deduce that $T_{v,w,\varphi,b}x=0$ if and only if $x=0$.\\

Since the operator $T_{v,w,\varphi,b}$ is injective, we can now easily check that
\begin{itemize}
\item for every $n\geq 0$, every $k\in (b_{n},b_{n+1})$,
\[T^{-1}_{v,w,\varphi,b}\ e_k= \frac{1}{w_{k}}\, e_{k-1};\]
\item $T^{-1}_{v,w,\varphi,b}e_0= -W_0\ e_{b_1-1}$;
\item for every $n\ge 1$,
\[
T^{-1}_{v,w,\varphi,b}\ e_{b_n}=-\sum_{m=0}^{m_n-1} \Big(\prod_{l=0}^mv_{\varphi^l(n)}\Big)\,\Big(\prod_{l=0}^{m+1}W_{\varphi^{l}(n)}\Big) e_{b_{\varphi^{m+1}(n)+1}-1}-W_n \ e_{b_{n+1}-1}\]
where $m_n=\min\{m\ge 0:\varphi^m(n)=0\}$.
\end{itemize}
%We remark that $T^{-1}$ is a continuous operator on $\ell_p(\N)$ since $w$ is bounded below, $\sup_{n\geq 0} |W_n|<\infty$ and 
%\[\sum_{n\ge 1}\sum_{m= 0}^{m_n-1} \Big(\prod_{l=0}^m|v_{\varphi^l(n)}|\Big)\,\Big(\prod_{l=0}^{m+1}|W_{\varphi^{l}(n)}|\Big)\le \sum_{n\ge 1}\sum_{m= 0}^{m_n-1} \frac{C}{2^n}\le \sum_{n\ge 1}\frac{m_nC}{2^n}<\infty\]
%since $m_n\le n$ for every $n\ge 1$.\\

 It remains to show that $T_{v,w,\varphi,b}$ is surjective. Let $z\in \ell^p(\mathbb{N})$. It suffices to show that the sequence $(T_{v,w,\varphi,b}^{-1}P_{[0,b_{n+1}-1]}z)_n$ is a Cauchy sequence, where $P_{[i,j]}z=\sum_{k=i}^jz_ke_k$. We have for every $N\ge n$,
\begin{align*}
&\|T_{v,w,\varphi,b}^{-1}P_{[0,b_{N+1}-1]}z-T_{v,w,\varphi,b}^{-1}P_{[0,b_{n+1}-1]}z\|\\
&\quad=\Big\|\sum_{k=b_{n+1}}^{b_{N+1}-1}z_kT_{v,w,\varphi,b}^{-1}e_k\Big\|\\
&\quad\le \Big\|\sum_{l=n+1}^N\sum_{k=b_l+1}^{b_{l+1}-1}z_k\frac{1}{w_{k}}\, e_{k-1}\Big\|+\Big\|\sum_{l=n+1}^Nz_{b_l}T_{v,w,\varphi,b}^{-1}e_{b_l}\Big\|\\
&\quad\le \frac{1}{\inf_k |w_k|}\|P_{[b_{n+1},b_{N+1}-1]}z\|\\
&\quad\quad+ \sum_{l=n+1}^N|z_{b_l}|  \left(
\sum_{m=0}^{m_l-1} \Big(\prod_{s=0}^m|v_{\varphi^s(l)}|\Big)\,\Big(\prod_{s=0}^{m+1}|W_{\varphi^s(l)}|\Big)\right)+ \left\|\sum_{l=n+1}^NW_lz_{b_l}e_{b_{l+1}-1}\right\|\\
&\quad\le \frac{1}{\inf_k |w_k|} \|P_{[b_{n+1},b_{N+1}-1]}z\|+\|z\|_{\infty}\sum_{l=n+1}^N\sum_{m=0}^{m_l-1}\frac{C}{2^l}+C\|P_{[b_{n+1},b_{N+1}-1]}z\|\\
&\quad\le   \big(\frac{1}{\inf_k |w_k|}+C\big)\|P_{[b_{n+1},b_{N+1}-1]}z\|+\|z\|_{\infty}\sum_{l=n+1}^N  \frac{m_lC}{2^l}\\
&\quad\le  \big(\frac{1}{\inf_k |w_k|}+C\big)\|P_{[b_{n+1},b_{N+1}-1]}z\|+\|z\|_{\infty}\sum_{l=n+1}^N\frac{Cl}{2^l}
\end{align*}
since $m_l\le l$ for every $l\ge 1$. The sequence $(T_{v,w,\varphi,b}^{-1}P_{[0,b_{n+1}-1]}z)_n$ is thus a Cauchy sequence and we deduce that $T_{v,w,\varphi,b}$ is surjective. It now follows from the open mapping theorem that $T^{-1}_{v,w,\varphi,b}$ is continuous.
\end{proof}

The counterexample that we will construct will be an operator of C-type with a specific structure that is called operator of $C^{+}$-type.

\begin{definition}[{\cite[Definition 6.6]{Monster}}]
An operator of C-type $T_{v,w,\varphi,b}$ is said to be an operator of $C^{+}$-type if for every integer $k\ge 1$,
\begin{enumerate}
 \item[-] $\varphi (n)=n-2^{k-1}$ for every $n\in[2^{k-1},2^{k})$, so that 
 $\varphi ([2^{k-1},2^{k}))=[0,2^{k-1})$;
 \item[-] the blocks $[b_{n},b_{n+1})$ with $n\in[2^{k-1},2^{k})$,
all have the same size, which we denote by $\Delta ^{(k)}$: 
\[b_{n+1}-b_{n}=
\Delta ^{(k)}\qquad\hbox{for every $n\in[2^{k-1},2^{k})$};
\]
\item[-] the sequence $v$ is constant on the interval $[2^{k-1},
2^{k})$: there exists $v^{(k)}$ such that 
\[ v_{n}=v^{(k)}\qquad\hbox{for every 
$n\in[2^{k-1},2^{k})$};
\]
\item[-] the sequences of weights $(w_{b_{n}+i})_{1\le i
<\Delta ^{(k)}}$ are independent of $n\in[2^{k-1},2^{k})$: there exists a 
sequence $(w_{i}^{(k)})_{1\le i<\Delta ^{(k)}}$ such that 
\[
w_{b_{n}+i}=w_{i}^{(k)} \quad\hbox{for every $1\le i<\Delta ^{(k)}$ and every 
$n\in[2^{k-1},2^{k})$.}
\]
\end{enumerate}
\end{definition}

We will now consider that $T$ is an operator of C$^{+}$-type such that for every $k\ge 1$,
\[
v^{(k)}=2^{-\tau^{(k)}} \quad \text{and}\quad 
w^{(k)}_i=
\begin{cases}
  \frac{1}{2} & \quad\text{if}\ \ 1\le i\le \delta^{(k)}\\
 2 & \quad\text{if}\ \ \delta^{(k)}<i\le 2\delta^{(k)}\\
 1 & \quad\text{if}\ \  2\delta^{(k)}\le i <\Delta^{(k)}-2\delta^{(k)}-2\eta^{(k)}\\
 \frac{1}{2} & \quad\text{if}\ \ \Delta^{(k)}-2\delta^{(k)}-2\eta^{(k)}\le i<\Delta^{(k)}-\delta^{(k)}-2\eta^{(k)}\\
 1 & \quad\text{if}\ \ \Delta^{(k)}-\delta^{(k)}-2\eta^{(k)}\le i<\Delta^{(k)}-\delta^{(k)}-\eta^{(k)}\\
 2& \quad\text{if}\ \ \Delta^{(k)}-\delta^{(k)}-\eta^{(k)}\le i<\Delta^{(k)}-\eta^{(k)}\\
 1& \quad\text{if}\ \ \Delta^{(k)}-\eta^{(k)}\le i<\Delta^{(k)}\\ 
\end{cases}
\]
where $(\tau ^{(k)})_{k\ge 1}$, $(\delta ^{(k)})_{k\ge 1}$ and $(\eta^{(k)})_{k\ge 1}$ are three strictly 
increasing sequences of positive integers satisfying $4\delta^{(k)}+2\eta^{(k)}<\Delta 
^{(k)}$ for every $k\ge 1$.

Since $W_n=1$ for every $n\ge 0$,
it follows from Proposition~\ref{invertible} that $T$ is invertible if 
\begin{equation}
\tau^{(k)}\ge 2^{k}
\label{tau}
\end{equation}
since we then get that for every $n\in [2^{k-1},2^k)$, 
$v_n = 2^{-\tau^{(k)}}\le 2^{-2^k}<\frac{1}{2^n}$.

We will discuss in the following sections our choice of weights $(w_{i}^{(k)})_{1\le i<\Delta ^{(k)}}$ and we will see that under some conditions on the parameters $\tau^{(k)}$, $\delta^{(k)}$, $\eta^{(k)}$ and $\Delta^{(k)}$, the operator $T$ is not $\mathcal{U}$-frequently hypercyclic whereas its inverse is $\mathcal{U}$-frequently hypercyclic.

\section{The inverse of $T$ is $\mathcal{U}$-frequently hypercyclic}

In \cite{Monster}, a criterion for $\mathcal{U}$-frequent hypercyclicity based on the study of periodic points has been given. We will use this one to determine under which conditions the inverse of $T$ is $\mathcal{U}$-frequently hypercyclic.

\begin{theorem}[{\cite[Theorem 5.14]{Monster}}]\label{Theorem 39} Let $T\in\mathfrak B(X)$. Assume that there exist a dense linear subspace 
$X_0$ of $X$ with $T(X_0)\subseteq X_0$ and $X_0\subseteq \emph{Per}(T)$, and a constant
$\alpha\in (0,1)$ such that the following property holds true: for every $x\in X_{0}$ and every $\varepsilon >0$, there exist 
$z\in X_{0}$ and $n\ge 1$ such that 
\begin{enumerate}
 \item [\emph{(1)}] $\|z\|<\varepsilon $;
 \item [\emph{(2)}] $\|T^{n+k}z-T^{k}x\|<\varepsilon $ for every
 $0\le k\le n\alpha $.
\end{enumerate}
Then $T$ is chaotic and $\mathcal{U}$-frequently hypercyclic.
\end{theorem}

We will apply this theorem to $T^{-1}$. We recall that every finite sequence is a periodic point of $T$ and thus a periodic point of $T^{-1}$. Moreover, it follows from Proposition~\ref{invertible} that
\[
T^{-1}_{v,w,\varphi,b}\ e_j=
\begin{cases}
 \frac{1}{w_{j}}\, e_{j-1} & \textrm{if}\ j\in (b_{n},b_{n+1}),\\
-\sum_{m=0}^{m_n-1} \Big(\prod_{l=0}^mv_{\varphi^l(n)}\Big)\, e_{b_{\varphi^{m+1}(n)+1}-1}-e_{
b_{n+1}-1} & \textrm{if}\ j=b_{n},\ n\ge 1,\\
 -e_{b_1-1}& \textrm{if}\ 
j=0
\end{cases}
\]
where if $j=b_n+i$ with $n\in [2^{k-1},2^k)$ and $i\in [1,\Delta^{(k)})$,
\[
\frac{1}{w_j}=
\begin{cases}
  2 & \quad\text{if}\ \ 1\le i\le \delta^{(k)}\\
 \frac{1}{2} & \quad\text{if}\ \ \delta^{(k)}<i\le 2\delta^{(k)}\\
 1 & \quad\text{if}\ \  2\delta^{(k)}\le i <\Delta^{(k)}-2\delta^{(k)}-2\eta^{(k)}\\
 2 & \quad\text{if}\ \ \Delta^{(k)}-2\delta^{(k)}-2\eta^{(k)}\le i<\Delta^{(k)}-\delta^{(k)}-2\eta^{(k)}\\
 1 & \quad\text{if}\ \ \Delta^{(k)}-\delta^{(k)}-2\eta^{(k)}\le i<\Delta^{(k)}-\delta^{(k)}-\eta^{(k)}\\
 \frac{1}{2}& \quad\text{if}\ \ \Delta^{(k)}-\delta^{(k)}-\eta^{(k)}\le i<\Delta^{(k)}-\eta^{(k)}\\
 1& \quad\text{if}\ \ \Delta^{(k)}-\eta^{(k)}\le i<\Delta^{(k)}\\ 
\end{cases}
\]

Let $n\in [2^{k-1},2^k)$. Roughly speaking, the first block of $2$ (for $1\le i\le \delta^{(k)}$) will allow that the action of $T^{-1}$ on the coordinate $e_{b_n+\delta^{(k)}}$ brings big coordinates in the blocks $[b_{\varphi^m(n)},b_{\varphi^m(n)+1})$, at least if $\delta^{(k)}$ is bigger than $\tau^{(k)}$. However, this approximation will be only acceptable when the effect of the coordinate $e_{b_n+\delta^{(k)}}$ in the block $[b_{n},b_{n+1})$ will be sufficiently small. This will happen after $\delta^{(k)}+\eta^{(k)}$ additional iterates thanks to the block of $1/2$ for $\Delta^{(k)}-\delta^{(k)}-\eta^{(k)}\le i<\Delta^{(k)}-\eta^{(k)}$ and will last during $\eta^{(k)}$ additional iterates. In other words, when we will apply Theorem~\ref{Theorem 39} to $T^{-1}$, the integer $n$ will be of the order of $2\delta^{(k)}+\eta^{(k)}$ and the good approximation will last during a sufficiently long time if $\eta^{(k)}$ is not too small compared to $\delta^{(k)}$. 

\begin{prop}\label{Prop1}
If $\lim \delta^{(k)}-\tau^{(k)}=\infty$ and if $\liminf_k\frac{\delta^{(k)}}{\eta^{(k)}}<\infty$, then $T^{-1}$ is $\mathcal{U}$-frequently hypercyclic.
\end{prop}
\begin{proof}
Let $X_{0}$ be the set of finite sequences, $x\in X_{0}$ and
 $\varepsilon >0$. We choose $k_0\geq1$ such that $x$ may be written as 
\[x=\sum_{l<2^{k_0}}\sum_{j=b_l}^{b_{l+1}-1}
x_{j}e_{j}.\]
Let $A>\max\{1,\liminf_k\frac{\delta^{(k)}}{\eta^{(k)}}\}$ and \[C=\sum_{l<2^{k_0}} 
\sum_{j=b_l}^{b_{l+1}-1} \sum_{m= 0}^{m_l-1} \Big(\prod_{s=0}^{m} v_{\varphi^s(l)}\Big)\, \!\!\!\!\!\!\!\!\!\;\;\;\;\;\;2^{\delta^{(k_0)}}  \|T^{-1}\|^{b_{l+1}-j-1}.\] We consider $c>0$ such that 
\begin{equation}
\label{c}
c\|x\|\sum_{s=0}^{k_0}(k_0+1)^{s+1}C^s<\varepsilon
\end{equation}
and $k_1> k_0$ such that $2\Delta^{(k_0)}<\frac{1}{6}\eta^{(k_1)}$, $\Delta^{(k_0)}<\delta^{(k_1)}$, $\delta^{(k_1)}/\eta^{(k_1)}<A$,
\begin{equation}
\label{k1}
\frac{2^{\Delta^{(k_0)}+\delta^{(k_0)}}}{2^{\delta^{(k_1)}-\tau^{(k_1)}}}<c\quad \text{and}\quad \left(\sum_{s=0}^{k_0}(k_0+1)^sC^{s}\right)\|x\|
\frac{2^{\delta^{(k_0)}}b_{2^{k_0}}}{2^{\delta^{(k_1)}-\tau^{(k_1)}}}<\varepsilon.
\end{equation}
We then let $n=\delta^{(k_1)}+2N\Delta^{(k_0)}$ where $N$ is the positive integer satisfying $\eta^{(k_1)}+\delta^{(k_1)}<2N\Delta^{(k_0)}\le \eta^{(k_1)}+\delta^{(k_1)}+2\Delta^{(k_0)}$ and we show that there exists $z\in X_0$ with $\|z\|<\varepsilon$ such that $\|T^{-n-m}z-T^{-m}x\|<\varepsilon$ for every $0\le m\le \frac{1}{6A}n$. It will then follow from Theorem~\ref{Theorem 39} that $T^{-1}$ is $\mathcal{U}$-frequently hypercyclic.\\

We first prove that for every $\tilde{k}\le k_0$, every $\tilde{x}=\sum_{l=2^{\tilde{k}-1}}^{2^{\tilde{k}}-1}\sum_{j=b_l}^{b_{l+1}-1}
\tilde{x}_{j}e_{j}$ (where, in the case of $\tilde{k}=0$, the sum $\sum_{l=2^{\tilde{k}-1}}^{2^{\tilde{k}}-1}$ means $\sum_{l=0}^{0}$), there exists $z$ such that $\|z\|\le c \|\tilde{x}\|$ and 
\begin{align*}
T^{-n}z&=\tilde{x}+\tilde{y}+\left[\sum_{l=2^{\tilde{k}-1}}^{2^{\tilde{k}}-1}\ 
\sum_{j=b_l}^{b_{l+1}-1}
\tilde{x}_{j}\
({v^{(k_1)}})^{-1}\!\!\!\!\!\!\!\!\!\;\;\;\;\;\;
\Bigl(\,\prod_{i=j+1}^{b_{l+1}-1}w_{i} \Bigr)\right.\\
&\quad\quad\quad
\left.\Bigl(\,\prod_{i=\Delta^{(k_1)}-b_{l+1}+j+1-2N\Delta^{(k_0)}}^{\Delta^{(k_1)}-1}w^{(k_1)}_{i} \Bigr)^{-1}
\ e_{b_{2^{k_1-1}+l+1}+j-b_{l+1}-2N\Delta^{(k_0)}}\right]
\end{align*} with $\|\tilde{y}\|\le C\|\tilde{x}\|$ and $\tilde{y}\in \text{span}\{e_j: 0\le j< b_{2^{\tilde{k}-1}}\}$ if $\tilde{k}\ge 1$ and $\tilde{y}=0$ if $\tilde{k}=0$.

To this end, we set
\[
z=-\sum_{l=2^{\tilde{k}-1}}^{2^{\tilde{k}}-1}\ 
\sum_{j=b_l}^{b_{l+1}-1}
\tilde{x}_{j}\
({v^{(k_1)}})^{-1}\!\!\!\!\!\!\!\!\!\;\;\;\;\;\;\Bigl(\prod_ {
i=1}^{\delta^{(k_1)}+j-b_{l+1}}w^{(k_1)}_{i}\Bigr)
\Bigl(\,\prod_{i=j+1}^{b_{l+1}-1}w_{i} \Bigr)
\ e_{b_{2^{k_1-1}+l}+\delta^{(k_1)}+j-{b_{l+1}}},
\]
which is well-defined since $b_{l+1}-j\le\Delta^{(\tilde{k})}\le\Delta^{(k_0)}<\delta^{(k_1)}$.
In view of our choice of weights $w$, it follows from \eqref{k1} that
\begin{align*}
\|z\|\le \|\tilde{x}\| 2^{\tau^{(k_1)}} 2^{-\delta^{(k_1)}+\Delta^{(\tilde{k})}}2^{\delta^{(\tilde{k})}}\le \frac{2^{\Delta^{(k_0)}+\delta^{(k_0)}}}{2^{\delta^{(k_1)}-\tau^{(k_1)}}}\|\tilde{x}\| \le c \|\tilde{x}\| 
\end{align*}
and by definition of $T^{-1}$, we have
\begin{align*}
T^{-\delta^{(k_1)}}z&=\tilde{x}+\tilde{y}+\left[\sum_{l=2^{\tilde{k}-1}}^{2^{\tilde{k}}-1}\ 
\sum_{j=b_l}^{b_{l+1}-1}
\tilde{x}_{j}\
({v^{(k_1)}})^{-1}\!\!\!\!\!\!\!\!\!\;\;\;\;\;\;
\Bigl(\,\prod_{i=j+1}^{b_{l+1}-1}w_{i} \Bigr)\right.\\
&\quad\quad\quad\quad\quad\quad \left.
\Bigl(\,\prod_{i=\Delta^{(k_1)}-b_{l+1}+j+1}^{\Delta^{(k_1)}-1}w^{(k_1)}_{i} \Bigr)^{-1}
\ e_{b_{2^{k_1-1}+l+1}+j-b_{l+1}}\right]
\end{align*}
where if $\tilde{k}=0$, we have $\tilde{y}=0$ and if  $\tilde{k}\ge 1$, we have
\begin{align*}
\tilde{y}&=\sum_{l=2^{\tilde{k}-1}}^{2^{\tilde{k}}-1}\ 
\sum_{j=b_l}^{b_{l+1}-1} \sum_{m=1}^{m_{2^{k_1-1}+l}-1} \left[\tilde{x}_{j}\Big(\prod_{s=1}^m v_{\varphi^s(2^{k_1-1}+l)}\Big)\right.\\
&\quad\quad\quad\quad\quad\quad \quad\quad\quad \quad\quad\quad \quad
\left.\Bigl(\,\prod_{i=j+1}^{b_{l+1}-1}w_{i} \Bigr)\ T^{-b_{l+1}+j+1}e_{b_{\varphi^{m+1}(2^{k_1-1}+l)+1}-1}\right]\\
&=\sum_{l=2^{\tilde{k}-1}}^{2^{\tilde{k}}-1}\ 
\sum_{j=b_l}^{b_{l+1}-1} \sum_{m=1}^{m_{2^{k_1-1}+l}-1} \tilde{x}_{j}\Big(\prod_{s=0}^{m-1} v_{\varphi^s(l)}\Big)
\Bigl(\,\prod_{i=j+1}^{b_{l+1}-1}w_{i} \Bigr)\ T^{-b_{l+1}+j+1}e_{b_{\varphi^{m}(l)+1}-1}\\
&=\sum_{l=2^{\tilde{k}-1}}^{2^{\tilde{k}}-1}\ 
\sum_{j=b_l}^{b_{l+1}-1} \sum_{m=0}^{m_{l}-1} \tilde{x}_{j}\Big(\prod_{s=0}^{m} v_{\varphi^s(l)}\Big)
\Bigl(\,\prod_{i=j+1}^{b_{l+1}-1}w_{i} \Bigr)\ T^{-b_{l+1}+j+1}e_{b_{\varphi^{m+1}(l)+1}-1}.
\end{align*}

Therefore, since $n=\delta^{(k_1)}+2N\Delta^{(k_0)}$, by using the periodicity of finite sequences, we have
\begin{align*}
T^{-n}z&=\tilde{x}+\tilde{y}+\left[\sum_{l=2^{\tilde{k}-1}}^{2^{\tilde{k}}-1}\ 
\sum_{j=b_l}^{b_{l+1}-1}
\tilde{x}_{j}\
({v^{(k_1)}})^{-1}\Bigl(\,\prod_{i=j+1}^{b_{l+1}-1}w_{i} \Bigr)\right.\\
&\quad\quad\left.
\Bigl(\,\prod_{i=\Delta^{(k_1)}-b_{l+1}+j+1-2N\Delta^{(k_0)}}^{\Delta^{(k_1)}-1}w^{(k_1)}_{i} \Bigr)^{-1}
\ e_{b_{2^{k_1-1}+l+1}+j-b_{l+1}-2N\Delta^{(k_0)}}\right]
\end{align*}
and by definition of $C$,
\begin{align*}
\|\tilde{y}\|&\le \|\tilde{x}\|\sum_{l=2^{\tilde{k}-1}}^{2^{\tilde{k}}-1}\ 
\sum_{j=b_l}^{b_{l+1}-1} \sum_{m=0}^{m_l-1} \Big(\prod_{s=0}^m v_{\varphi^s(l)}\Big)\, \!\!\!\!\!\!\!\!\!\;\;\;\;\;\;2^{\delta^{(\tilde{k})}}  \|T^{-1}\|^{b_{l+1}-j-1}\\
&\le \|\tilde{x}\|\sum_{l=0}^{2^{k_0}-1}\ 
\sum_{j=b_l}^{b_{l+1}-1} \sum_{m=0}^{m_l-1} \Big(\prod_{s=0}^{m} v_{\varphi^s(l)}\Big)\, \!\!\!\!\!\!\!\!\!\;\;\;\;\;\;2^{\delta^{(k_0)}}  \|T^{-1}\|^{b_{l+1}-j-1}=C\|\tilde{x}\|.
\end{align*}

We are now able to construct a vector $z$ such that $\|z\|<\varepsilon$ and such that $\|T^{-n-m}z-T^{-m}x\|<\varepsilon$ for every $0\le m\le \frac{1}{6A}n$.

Let $u^0=R_{k_0}(x)$ where $R_k(x)=\sum_{l=2^{k-1}}^{2^k-1}\sum_{j=b_l}^{b_{l+1}-1}x_je_j$. We know that there exist $z^0$ and $y^0$ such that
$\|z^0\|\le c \|u^0\|$, such that 
\begin{align*}
T^{-n}z^0&=u^0+y^0\\
&\quad+\sum_{l=2^{k_0-1}}^{2^{k_0}-1}\ 
\sum_{j=b_l}^{b_{l+1}-1}
u^0_{j}\
({v^{(k_1)}})^{-1}\!\!\!\!\!\!\!\!\!\;\;\;\;\;\;
\Bigl(\,\prod_{i=j+1}^{b_{l+1}-1}w_{i} \Bigr)\\
&\quad\quad
\Bigl(\,\prod_{i=\Delta^{(k_1)}-b_{l+1}+j+1-2N\Delta^{(k_0)}}^{\Delta^{(k_1)}-1}w^{(k_1)}_{i} \Bigr)^{-1}
\ e_{b_{2^{k_1-1}+l+1}+j-b_{l+1}-2N\Delta^{(k_0)}}
\end{align*}
and such that $\|y^{0}\|\le C\|u^0\|$ and $y^{0}\in  \text{span}\{e_j: 0\le j< b_{2^{k_0-1}}\}$. By repeating, this argument, we can obtain two families $(z^k)_{0\le k\le k_0}$ and $(y^k)_{0\le k\le k_0}$ such that if we let $u^k=R_{k_0-k}(x)-\sum_{s=0}^{k-1}R_{k_0-k}y^s$ then
$\|z^k\|\le c \|u^k\|$ and
\begin{align*}
T^{-n}z^k&=u^k+y^{k}\\
&\quad+\sum_{l=2^{k_0-k-1}}^{2^{k_0-k}-1}\ 
\sum_{j=b_l}^{b_{l+1}-1}
u^k_{j}\
({v^{(k_1)}})^{-1}\!\!\!\!\!\!\!\!\!\;\;\;\;\;\;
\Bigl(\,\prod_{i=j+1}^{b_{l+1}-1}w_{i} \Bigr)\\
&\quad\quad\Bigl(\,\prod_{i=\Delta^{(k_1)}-b_{l+1}+j+1-2N\Delta^{(k_0)}}^{\Delta^{(k_1)}-1}w^{(k_1)}_{i} \Bigr)^{-1}
\ e_{b_{2^{k_1-1}+l+1}+j-b_{l+1}-2N\Delta^{(k_0)}}
\end{align*}
and such that $\|y^{k}\|\le C\|u^k\|$, $y^{k}\in  \text{span}\{e_j: 0\le j< b_{2^{k_0-k-1}}\}$ and $y^{k_0}=0$.\\

We remark that 
\begin{align*}
\sum_{k=0}^{k_0}u^k&=\sum_{k=0}^{k_0}\left(R_{k_0-k}(x)-\sum_{s=0}^{k-1}R_{k_0-k}y^s\right)\\
&=x-\sum_{k=0}^{k_0}\sum_{s=0}^{k-1}R_{k_0-k}y^s\\
&=x-\sum_{s=0}^{k_0-1}\sum_{k=s+1}^{k_0}R_{k_0-k}y^s=x-\sum_{s=0}^{k_0-1}y^s
\end{align*}
and thus if we let $z=\sum_{k=0}^{k_0}z^k$, we get 
\begin{align*}
T^{-n}z=x+&\sum_{k=0}^{k_0}\sum_{l=2^{k_0-k-1}}^{2^{k_0-k}-1}\ 
\sum_{j=b_l}^{b_{l+1}-1}
u^k_{j}\
({v^{(k_1)}})^{-1}\!\!\!\!\!\!\!\!\!\;\;\;\;\;\;
\Bigl(\,\prod_{i=j+1}^{b_{l+1}-1}w_{i} \Bigr)\\
&\quad
\Bigl(\,\prod_{i=\Delta^{(k_1)}-b_{l+1}+j+1-2N\Delta^{(k_0)}}^{\Delta^{(k_1)}-1}w^{(k_1)}_{i} \Bigr)^{-1}
\ e_{b_{2^{k_1-1}+l+1}+j-b_{l+1}-2N\Delta^{(k_0)}}.
\end{align*}

We start by proving that $\|z\|<\varepsilon$. Since $\|u^0\|\le\|x\|$ and since for every $1\le k\le k_0$, 
\[\|u^k\|\le \|x\|+\sum_{s=0}^{k-1}\|y^s\|\le \|x\|+\sum_{s=0}^{k-1}C\|u^s\|,\]
we can deduce that for every $0\le k\le k_0$, $\|u^k\|\le \left(\sum_{s=0}^{k}(k+1)^sC^s\right)\|x\|$ and
it follows from \eqref{c} that 
\[\|z\|\le \sum_{k=0}^{k_0}\|z_k\|\le c\sum_{k=0}^{k_0}\|u_k\|\le c\|x\|(k_0+1)\sum_{s=0}^{k_0}(k_0+1)^sC^s<\varepsilon.\]
Moreover, for every $0\le m\le \frac{1}{6A} n$, we have
\begin{align*}
&\|T^{-n-m}z-T^{-m}x\|\\
&\quad\le \left\|\sum_{k=0}^{k_0}\sum_{l=2^{k_0-k-1}}^{2^{k_0-k}-1}\ 
\sum_{j=b_l}^{b_{l+1}-1}
u^k_{j}\
({v^{(k_1)}})^{-1}\!\!\!\!\!\!\!\!\!\;\;\;\;\;\;
\Bigl(\,\prod_{i=j+1}^{b_{l+1}-1}w_{i} \Bigr)\right.\\
&\quad\quad\left.
\Bigl(\,\prod_{i=\Delta^{(k_1)}-b_{l+1}+j+1-2N\Delta^{(k_0)}}^{\Delta^{(k_1)}-1}w^{(k_1)}_{i} \Bigr)^{-1}
\ T^{-m}e_{b_{2^{k_1-1}+l+1}+j-b_{l+1}-2N\Delta^{(k_0)}}\right\|\\
&\quad\le \left\|\sum_{k=0}^{k_0-1}\sum_{l=2^{k_0-k-1}}^{2^{k_0-k}-1}\ 
\sum_{j=b_l}^{b_{l+1}-1} \left(\sum_{s=0}^{k_0}(k_0+1)^sC^{s}\right)\|x\|
2^{\tau^{(k_1)}}2^{\delta^{(k_0)}}\right.\\
&\quad\quad\left.\Bigl(\,\prod_{i=\Delta^{(k_1)}-b_{l+1}-m+j+1-2N\Delta^{(k_0)}}^{\Delta^{(k_1)}-1}w^{(k_1)}_{i} \Bigr)^{-1}e_{b_{2^{k_1-1}+l+1}+j-b_{l+1}-2N\Delta^{(k_0)}-m}\right\|\\
&\quad\le \sum_{k=0}^{k_0-1}\sum_{l=2^{k_0-k-1}}^{2^{k_0-k}-1}\ 
\sum_{j=b_l}^{b_{l+1}-1} \left(\sum_{s=0}^{k_0}(k_0+1)^sC^{s}\right)\|x\|
2^{\tau^{(k_1)}}2^{\delta^{(k_0)}}2^{-\delta^{(k_1)}}\\
&\quad\le \left(\sum_{s=0}^{k_0}(k_0+1)^sC^{s}\right)\|x\|
\frac{2^{\delta^{(k_0)}}b_{2^{k_0}}}{2^{\delta^{(k_1)}-\tau^{(k_1)}}}<\varepsilon\quad\text{by \eqref{k1}}
\end{align*}
where we could replace $\Bigl(\,\prod_{i=\Delta^{(k_1)}-b_{l+1}-m+j+1-2N\Delta^{(k_0)}}^{\Delta^{(k_1)}-1}w^{(k_1)}_{i} \Bigr)^{-1}$ by $2^{-\delta^{(k_1)}}$ because for every $0\le k<k_0$, every $2^{k_0-k-1}\le l<2^{k_0}-k$, every $b_l\le j< b_{l+1}$, we have
\[\Delta^{(k_1)}-b_{l+1}-m+j+1-2N\Delta^{(k_0)}\le \Delta^{(k_1)}-2N\Delta^{(k_0)}\le \Delta^{(k_1)}-\delta^{(k_1)}-\eta^{(k_1)}\]
by choice of $N$ and 
\begin{align*}
&\Delta^{(k_1)}-b_{l+1}-m+j+1-2N\Delta^{(k_0)}\\
&\quad\ge \Delta^{(k_1)}-b_{l+1}-\frac{1}{6A} n+b_l-2N\Delta^{(k_0)}\\
&\quad\ge \Delta^{(k_1)}-\Delta^{(k_0)}-\frac{1}{6A}(\delta^{(k_1)}+2N\Delta^{(k_0)})-2N\Delta^{(k_0)}\\
&\quad\ge \Delta^{(k_1)}-\frac{1}{6}\eta^{(k_1)}-\frac{1}{6A}\delta^{(k_1)}-(1+\frac{1}{6A})(\eta^{(k_1)}+\delta^{(k_1)}+2\Delta^{(k_0)})\\
&\quad\ge \Delta^{(k_1)}-\delta^{(k_1)}-\frac{7}{6}\eta^{(k_1)}-\frac{2}{6A}\delta^{(k_1)}-\frac{1}{6A}\eta^{(k_1)}-(1+\frac{1}{6A})\frac{1}{6}\eta^{(k_1)}\\
&\quad\ge \Delta^{(k_1)}-\delta^{(k_1)}-\Big(\frac{7}{6}+\frac{2}{6}+\frac{1}{6A}+\frac{1}{6}+\frac{1}{36A}\Big)\eta^{(k_1)}\quad \text{since $\frac{\delta^{(k_1)}}{\eta^{(k_1)}}<A$}\\
&\quad>\Delta^{(k_1)}-\delta^{(k_1)}- 2\eta^{(k_1)}.
\end{align*}
We have thus succeeded in constructing  a vector $z\in X_0$ with $\|z\|<\varepsilon$ such that $\|T^{-n-m}z-T^{-m}x\|<\varepsilon$ for every $0\le m\le \frac{1}{6A}n$ and we can therefore conclude that $T^{-1}$ is $\mathcal{U}$-frequently hypercyclic by using Theorem~\ref{Theorem 39}.
\end{proof}

Before showing in the next section that under some additional conditions, $T$ is not $\mathcal{U}$-frequently hypercyclic, we try to motivate our choice of weights $w$. If we look at the previous proof, we can already remark that
\begin{itemize}
\item the first block of $1/2$ of $w$ ($1\le i\le \delta^{(k)}$), which gives a block of $2$ if we consider $1/w$, allowed us to construct a vector $z$ with small norm;
\item  the last block of $2$ of $w$ ($\Delta^{(k)}-\delta^{(k)}-\eta^{(k)}\le i< \Delta^{(k)}-\eta^{(k)}$), which gives a block of $1/2$ if we consider $1/w$, allowed us to approach $x$;
\item the second block of $1$ ($\Delta^{(k)}-\delta^{(k)}-2\eta^{(k)}\le i< \Delta^{(k)}-\delta^{(k)}-\eta^{(k)}$) allowed us to follow the orbit of $x$ during $\eta^{(k)}$ iterates.
\end{itemize}
Now, if we want to have a chance that $T$ is not $\mathcal{U}$-frequently hypercyclic, it is necessary that Theorem~\ref{Theorem 39} does not apply to $T$. If we proceed as for $T^{-1}$, we can remark that roughly speaking, if $n\in [2^{k-1},2^k)$, the last block of $1$ of $w$ ($\Delta^{(k)}-\eta^{(k)}\le i< \Delta^{(k)}$) induces that we have to wait at least $\eta^{(k)}$ iterates before a small coordinate in $[b_n,b_{n+1})$ brings big coordinates in the block $[b_{\varphi(n)},b_{\varphi(n)+1})$ under the action of $T$. Moreover, we will get a good approximation during at most $2\delta^{(k)}$ iterates because of block of $2$ for $\delta^{(k)}<i\le 2\delta^{(k)}$. In view of Theorem~\ref{Theorem 39}, we can then hope that if $\delta^{(k)}/\eta^{(k)}$ tends to $0$, $T$ may not be $\mathcal{U}$-frequently hypercyclic. These two facts have motivated our choice to end $w$ with a big block of $1$ and to put a block of $2$ directly after the first block of $\frac{1}{2}$. Finally, the second block of $1/2$ ($\Delta^{(k)}-2\delta^{(k)}-2\eta^{(k)}\le i< \Delta^{(k)}-\delta^{(k)}-2\eta^{(k)}$) was added to have $W_n=1$ and thus an invertible operator if $v_n$ is sufficiently small.

\section{$T$ is not $\mathcal{U}$-frequently hypercyclic}
We now show that the operator $T$ is not $\mathcal{U}$-frequently hypercyclic under convenient conditions on $\eta^{(k)}$, $\delta^{(k)}$, $\tau^{(k)}$ and $\Delta^{(k)}$. Conditions implying that an operator of C-type is not $\mathcal{U}$-frequently hypercyclic have been developed in \cite{Monster}. Unfortunately, these conditions are not sufficiently general to be applied at our operator $T$.

We start by generalizing Lemma 6.11 in \cite{Monster}. Given $x\in \ell^{p}(\N)$ and $l\ge 0$, we will denote in this section
\[P_lx=\sum_{i=b_l}^{b_{l+1}-1}x_ie_i\quad\text{and}\quad X=\sum_{l=0}^{\infty}\sum_{i=b_{l}}^{b_{l+1}-1}\Big(\prod_{s=i+1}^{b_{l+1}-1}w_s\Big)\,x_{i}e_{i}.\]

\begin{lemma}\label{Theorem 48}
 Let $T$ be an operator of C-type on $\ell^{p}(\N)$ and $x\in\ell^p(\N)\setminus\{0\}$. Suppose that there exist

\begin{enumerate}
 \item[-] a constant $C>0$,
 \item[-] a non-increasing sequence $(\beta _{l})_{l\ge 1}$ of positive real 
numbers with ${\sum_{l\ge 1}\sqrt{\beta _{l}}\le 1}$,
\item[-] a non-decreasing sequence $(N_{l})_{l\ge 1}$ of positive integers tending to infinity,
\item[-] a sequence $(R_l)_{l\ge 1}$ with $N_l\le R_l<\inf_{j\in \bigcup_{k\ge 1} \varphi^{-k}(l)}N_j$ and $R_l< b_{l+1}-b_l$,
\item[-] a sequence $(L_l)_{l\ge 1}$ with $(\frac{L_l}{N_l})$ tending to $0$,
\end{enumerate}
such that the following conditions are satisfied:
\smallskip
\begin{enumerate}
 \item[\emph{(1)}] $\|P_{n}x\|\le \|P_n X\|$ for every $n\ge 0$;
 \item[\emph{(2)}] $\sup\limits_{j\ge 0}\|P_{n}\,T^{\,j}P_{l}\,x\|\le 
C\beta _{l}\|P_lX\|$ for every $l\ge 1$ and every $0\le n<l$;
 \item[\emph{(2')}] $\|P_{n}\,T^{\,j}P_{l}\,x\|\le 
C\beta _{l}\|P_{[b_{l+1}-j,b_{l+1})}X\|$ for every $l\ge 1$, every $j\le R_l$ and every $0\le n<l$;
\item[\emph{(3)}] $\sup\limits_{0\le j\le N_{l}}\|P_{n}\,T^{\,j}P_{l}\,x\|\le C\beta _{l}\|P_{l}x\|$
for every $l\ge 1$ and every $0\le n<l$;
\item[\emph{(4)}] $\sup\limits_{j\ge 0}\sum\limits_{l>n}\|P_{n}\,T^{\,j}P_{l}x\|>C\|P_nX\|$ for every $n\ge 0$;
\item[\emph{(5)}] for every $l\ge 1$, every $N_l\le j\le R_l$, every $j+L_l\le k\le R_l$, 
\[\|P_l T^{k}P_{[b_{l+1}-j,b_{l+1})}x\|\ge \|P_{[b_{l+1}-j,b_{l+1})}X\|.\]
\end{enumerate}
\smallskip
Then there exists $\varepsilon >0$ such that
\[
\underline{\vphantom{p}\textrm{\emph{dens}}}\ \mathcal{N}_{T} \bigl(x,B(0,\varepsilon )^c\bigr)\ge 
\liminf_{l\to\infty }\ \inf_{k\ge R_{l}}\ \dfrac{\ \#\bigl\{0\le 
j\le 
k\,:\, \|P_{l}\,T^{\,j}P_{l}\,x\|\ge 2C\|P_l X\|\bigr\}}{k+1}.
\]
\end{lemma}
\begin{proof}
Let $x\in\ell^p(\N)\setminus\{0\}$. If $x$ is periodic then there exists $\varepsilon 
>0 $ such that 
\[
\underline{\vphantom{p}\textrm{dens}}\ \mathcal{N}_{T}(x,B(0,\varepsilon )^c)=1.
\]
Without loss of generality, we can thus assume that $x$ is not periodic. In particular, $x-P_{0}x\ne 0$ and we can find $l_{0}\ge 1$ such that 
\[
\|P_{l_{0}}\,x\|\ge \sqrt{\beta _{l_{0}}}\,\|x-P_{0}\,x\|
\]
since $\sum_{l\ge 1}\sqrt{\beta _{l}}\le 1$.

Let $\varepsilon<\min\{C,1/2\}\|P_{l_0}X\|$.
There exists a  strictly increasing  sequence of integers 
$(l_{n})_{n\ge 1}$ such that if we set 
\[ j_{n-1}:=\min\,\biggl\{j\ge \,
0\,:\,\sum_{l>l_{n-1}}\|P_{l_{n-1}}T^{j}P_{l}\,x\|>C\|P_{l_{n-1}}X\|\biggr\}, \]
then
\[j_{n-1}>N_{l_{n}}\quad l_n\in \bigcup_{k\ge 1}\varphi^{-k}(l_{n-1})\quad\textrm{and}\quad 
\|P_{l_{n}}X\|\ge\dfrac{1}{\sqrt{\beta 
_{l_{n}}}}\,\|P_{l_{n-1}}X\|\quad \textrm{for every } n\ge 1
\]
and, in addition, if $j_{n-1}\le R_{l_n}$ then 
\begin{equation}
\label{Equation 17}
\|P_{[b_{l_n+1}-j_{n-1},b_{l_n+1})}X\| \ge \dfrac{1}{\sqrt{
\beta 
_{l_n}}}\, \|P_{l_{n-1}}X\|.
\end{equation}
Indeed, if we assume that $l_{1},\dots,l_{n-1}$ have been chosen then $j_{n-1}$ is well-defined in view of assumption (4) and we can consider for $l_{n}$ an index satisfying
\begin{equation*}
 \|P_{l_{n-1}}T^{\,j_{n-1}}P_{l_{n}}\,x\|>C\sqrt{\beta_{l_{n}}}\|P_{l_{n-1}}X\|.
\end{equation*}
In particular, we have $l_n\in \bigcup_{k\ge 1}\varphi^{-k}(l_{n-1})$.
Moreover, by (2), we get
\[
\|P_{l_n}X\|\ge \dfrac{1}{C\beta_{l_{n}}}\|P_{l_{n-1}}T^{\,j_{n-1}}P_{l_{n}}\,x\|\ge \dfrac{1}{\sqrt{\beta_{l_n}}}\,\|P_{l_{n-1}}X\|
\]
and if $j_{n-1}\le R_{l_n}$, we get by (2')
\[ \|P_{[b_{l_n+1}-j_{n-1},b_{l_n+1})}X\| \ge \dfrac{1}{C\beta_{l_{n}}}\|P_{l_{n-1}}T^{\,j_{n-1}}P_{l_{n}}\,x\|\ge \dfrac{1}{\sqrt{\beta_{l_n}}}\, \|P_{l_{n-1}}X\|. \]
On the other hand, for every $0\le j\le N_{l_{n}}$, it follows from (1) and (3) that
\begin{align*}
\|P_{l_{n-1}}T^{\,j}P_{l_{n}}\,x\|&\le C\beta _{l_{n}}\,\|P_{l_{n}}
\,x\|\le C\beta _{l_{n}}\,\|x-P_{0}x\|\le \dfrac{C\beta 
_{l_{n}}}{\sqrt{\beta 
 _{l_{0}}}}\,\|P_{l_{0}}\,x\|\\
 &\le \dfrac{C\beta _{l_{n}}}{\sqrt{\beta 
 _{l_{0}}}}\,\|P_{l_{0}}X\|\le C\sqrt{\beta_{l_n}}\,\|P_{l_{0}}X\|\le C\sqrt{\beta_{l_n}}\,\|P_{l_{n-1}}X\|
\end{align*}
since for every $k\le n-1$, $\|P_{l_{k}}X\|\ge\dfrac{1}{\sqrt{\beta 
_{l_{k}}}}\,\|P_{l_{k-1}}X\|\ge \|P_{l_{k-1}}X\|$ and thus $\|P_{l_{n-1}}X\|\ge \|P_{l_{0}}X\|$. This implies that $j_{n-1}>N_{l_{n}}$.
\par\smallskip
For any integer $s\geq 1$, we denote by $n_{s}$ the smallest integer such that $s$ belongs to $ 
[j_{n_{s}-1},j_{n_{s}})$. We remark that $n_{s}$ tends to infinity as $s$ tends to 
infinity. We divide the study of the ratio $\dfrac{\# \bigl\{0\le j\le s\,:\,\|T^{\,j}x\|\ge \varepsilon\bigr\}}{s+1}$ into two cases: $R_{l_{n_s}}\le s$ or $s<R_{l_{n_s}}$.

\begin{itemize}
\item Case 1: $R_{l_{n_s}}\le s$.\\
 Since for every $j\ge 0$ and $n\ge 0$,
\[
\smash[b]{\|T^{\,j}\,x\|\ge\|P_{n}\,T^{\,j}\,x\|\ge 
\|P_{n}\,T^{\,j}P_{n}\,x\|-
\sum_{l>n}\|P_{n}\,T^{\,j}P_{l}\,x\|},
\]
we have by definition of $j_n$
\[
\|T^{\,j}x\|\ge \|P_{l_{n}}T^{\,j}P_{l_{n}}\,x\|-C\,\|P_{l_{n}}X\|\quad  
\textrm{for every}\ 0\le j< j_{n}.
\]
Therefore, for every $s\ge 1$, we have 
\begin{align*}
\bigl\{0\le j\le s\,:\,\|T^{\,j}x\|\ge \varepsilon\bigr\}&\supset
\bigl\{0\le j\le s\,:\,\|T^{\,j}x\|\ge C\|P_{l_{0}}X\|\bigr\}\\
&\supset \bigl\{0\le j\le s\,:\,\|T^{\,j}x\|\ge C\|P_{l_{n_{s}}}X\|\bigr\}\\
&\supset \bigl\{0\le j\le s\,:\,\|P_{l_{n_{s}}}T^{\,j}P_{l_{n_{s}}}\,x\|
\ge 2C\|P_{l_{n_{s}}}X\|\bigr\}
\end{align*}
since $j\le s< j_{n_s}$. It follows that if $R_{l_{n_s}}\le s$ then
\begin{align*}
\dfrac{\# \bigl\{0\le j\le s\,:\,\|T^{\,j}x\|\ge \varepsilon\bigr\}}{s+1}&\ge 
\dfrac{\ \#\bigl\{0\le 
j\le 
s\,:\, \|P_{l_{n_{s}}}\,T^{\,j}P_{l_{n_{s}}}\,x\|\ge 
2C\|P_{l_{n_{s}}}X\|\bigr\}}{s+1}\\
&\ge \inf_{k\ge R_{l_{n_s}}}\dfrac{\ \#\bigl\{0\le 
j\le 
k\,:\, \|P_{l_{n_{s}}}\,T^{\,j}P_{l_{n_{s}}}\,x\|\ge 
2C\|P_{l_{n_{s}}}X\|\bigr\}}{k+1}.
\end{align*}
\item Case 2: $s<R_{l_{n_s}}$.\\
It follows from (3) that for every $j_{n_s-1}\le j\le s$
\begin{align*}
\|T^jx\|&\ge \|P_{l_{n_s}}T^jP_{l_{n_s}}x\|-\sum_{l\in \bigcup_{k\ge 1} \varphi^{-k}(l_{n_s})}\|P_{l_{n_s}}T^jP_lx\|\\
&\ge \|P_{l_{n_s}}T^jP_{[b_{l_{n_s}+1}-j_{n_{s}-1},b_{l_{n_s}+1})}x\|-\sum_{l\in \bigcup_{k\ge 1} \varphi^{-k}(l_{n_s})}C\beta_l\|P_lx\|
\end{align*}
since $j\le s<R_{l_{n_s}}<\inf_{l\in \bigcup_{k\ge 1} \varphi^{-k}(l_{n_s})}N_l$. Since $\inf_{l\in \bigcup_{k\ge 1} \varphi^{-k}(l_{n_s})}N_l$ tends to infinity and $\sum_{l\ge 1}\beta_l<\infty$, there exists $s_0$ such that for every $s\ge s_0$, if $s<R_{l_{n_s}}$, we have for every $j_{n_s-1}\le j\le s$
\[\|T^jx\|\ge \|P_{l_{n_s}}T^jP_{[b_{l_{n_s}+1}-j_{n_s-1},b_{l_{n_s}+1})}x\|-\varepsilon.\]
Therefore, for every $s\ge s_0$ satisfying $s<R_{l_{n_s}}$, since $R_{l_{n_s-1}}< N_{l_{n_s}}\le j_{n_s-1}\le s <R_{l_{n_s}}$, we have
\begin{align*}
&\frac{\#\{0\le j\le s: \|T^jx\|\ge \varepsilon\}}{s+1}\\
&\quad\ge \frac{\#\{0\le j< j_{n_s-1}: \|T^jx\|\ge \varepsilon\}}{s+1}+\frac{\#\{j_{n_s-1}\le j\le s: \|T^jx\|\ge \varepsilon\}}{s+1}\\
&\quad\ge \dfrac{\ \#\bigl\{ 
0\le j< j_{n_s-1}\,:\, \|P_{l_{n_s-1}}\,T^{\,j}P_{l_{n_s-1}}\,x\|\ge 
2C\|P_{l_{n_{s-1}}}X\|\bigr\}}{s+1}\\
&\quad\quad+\frac{\#\{j_{n_s-1}\le j\le s: \|P_{l_{n_s}}T^jP_{[b_{l_{n_s}+1}-j_{n_s-1},b_{l_{n_s}+1})}x\|\ge 2\varepsilon\}}{s+1}\\
&\quad\ge \frac{j_{n_s-1}}{s+1}\inf_{k\ge R_{l_{n_s-1}}}\dfrac{\ \#\bigl\{0\le 
j\le 
k\,:\, \|P_{l_{n_s-1}}\,T^{\,j}P_{l_{n_s-1}}\,x\|\ge 
2C\|P_{l_{n_s-1}}X\|\bigr\}}{k+1}\\
&\quad\quad+\frac{\#\{j_{n_s-1}+L_{l_{n_{s}}}\le j\le s: \|P_{[b_{l_{n_s}+1}-j_{n_s-1},b_{l_{n_s}+1})}\,X\|\ge 2\varepsilon\}}{s+1} \quad \text{(by (5))}\\
&\quad\ge \frac{j_{n_s-1}}{s+1}\inf_{k\ge R_{l_{n_s-1}}}\dfrac{\ \#\bigl\{0\le 
j\le 
k\,:\, \|P_{l_{n_s-1}}\,T^{\,j}P_{l_{n_s-1}}\,x\|\ge 
2C\|P_{l_{n_s-1}}X\|\bigr\}}{k+1}\\
&\quad\quad+\frac{\#\{j_{n_s-1}+L_{l_{n_{s}}}\le j\le s: \frac{\|P_{l_{n_s-1}}\,X\|}{\sqrt{\beta_{l_{n_s}}}}\ge 2\varepsilon\}}{s+1} \quad\text{(by \eqref{Equation 17})}\\
&\quad\ge \frac{j_{n_s-1}}{s+1}\inf_{k\ge R_{l_{n_s-1}}}\dfrac{\ \#\bigl\{0\le 
j\le 
k\,:\, \|P_{l_{n_s-1}}\,T^{\,j}P_{l_{n_s-1}}\,x\|\ge 
2C\|P_{l_{n_s-1}}X\|\bigr\}}{k+1}\\
&\quad\quad+\frac{s-j_{n_s-1}-L_{l_{n_{s}}}}{s+1} \quad \text{(since $\|P_{l_{n_s-1}}\,X\|\ge \|P_{l_{0}}\,X\|>2\varepsilon$)}\\
&\quad\ge \frac{s-L_{l_{n_s}}}{s+1}\inf_{k\ge R_{l_{n_s-1}}}\dfrac{\ \#\bigl\{0\le 
j\le 
k\,:\, \|P_{l_{n_s-1}}\,T^{\,j}P_{l_{n_s-1}}\,x\|\ge 
2C\|P_{l_{n_s-1}}X\|\bigr\}}{k+1}\\
&\quad\ge \frac{N_{l_{n_s}}-L_{l_{n_s}}}{N_{l_{n_s}}+1}\inf_{k\ge R_{l_{n_s-1}}}\dfrac{\ \#\bigl\{0\le 
j\le 
k\,:\, \|P_{l_{n_s-1}}\,T^{\,j}P_{l_{n_s-1}}\,x\|\ge 
2C\|P_{l_{n_s-1}}X\|\bigr\}}{k+1}\\
&\quad\ge \left(\inf_{l>l_{n_s-1}}\frac{N_{l}-L_{l}}{N_{l}+1}\right)\inf_{k\ge R_{l_{n_s-1}}}\dfrac{\ \#\bigl\{0\le 
j\le 
k\,:\, \|P_{l_{n_s-1}}\,T^{\,j}P_{l_{n_s-1}}\,x\|\ge 
2C\|P_{l_{n_s-1}}X\|\bigr\}}{k+1}.
\end{align*}
 \end{itemize}
Since $l_{n_s}\to\infty$ as $s\to\infty$ and $L_l/N_l\to 0$ as $l\to \infty$, we deduce from two cases that
\begin{align*}
&\underline{\vphantom{p}\textrm{dens}}\ \mathcal{N}_{T}
\bigl(x,B(0,\varepsilon )^c\bigr)\\
&\quad\ge 
\liminf_{l\to\infty }\ \left[\left(\inf_{l'>l}\frac{N_{l'}-L_{l'}}{N_{l'}+1}\right)\inf_{k\ge R_{l}}\ \dfrac{\ \#\bigl\{0\le 
j\le 
k\,:\, \|P_{l}\,T^{\,j}P_{l}\,x\|\ge 
2C\|P_{l}X\|\bigr\}}{k+1}\right]\\
&\quad=\liminf_{l\to\infty }\inf_{k\ge R_{l}}\ \dfrac{\ \#\bigl\{0\le 
j\le 
k\,:\, \|P_{l}\,T^{\,j}P_{l}\,x\|\ge 
2C\|P_{l}X\|\bigr\}}{k+1}.
\end{align*}
\end{proof}

In order to apply Lemma~\ref{Theorem 48} to our operator $T$, we will use the following two propositions.
\begin{prop}[{\cite[Proposition 6.12]{Monster}}]\label{Proposition 50}
 Let $T$ be an operator of C-type on $\ell^{p}(\N)$ and  
 let $(C_{n})_{n\ge 
0}$ be a sequence of positive numbers with $0<C_n<1$. Assume that 
\[
|v_n|\ .\sup_{j\in 
[b_{\varphi(n)},b_{\varphi(n)+1})}\ \Bigl(
\prod_{s=b_{\varphi(n)}+1}^{j}|w_{s}|\Bigr)\le C_{n}\quad\hbox{  for every $n\ge 1$. }
\]
 Then, for any $x\in\ell^p(\N)$, 
we have for every $l\ge 1$ and 
every $0\le n<l$,
\begin{enumerate}
 \item [\rm {(1)}]$\quad
 {\sup_{j\ge 0}\ \|P_{n}T^{\,j}P_{l}\,x\|\le 
C_{l}\,(b_{l+1}-b_{l})^{\frac{p-1}{p}}\ \|P_l X\|}
$
\par\smallskip \noindent and
\item[\rm {(2)}]
$
\quad {\sup_{j\le N}\|P_{n}T^{\,j}P_{l}\,x\|\le 
C_{l}\,(b_{l+1}-b_{l})^{\frac{p-1}{p}}\|P_{[b_{l+1}-N,b_{l+1})}\,X\|}
$
\par\smallskip \noindent for every $1\le N\le b_{l+1}-b_{l}$.
\end{enumerate}
\end{prop}

\begin{prop}[{\cite[Proposition 7.13]{Monster}}]
\label{Proposition 51}
 Let $T$ be an operator of C-type on $\ell^{p}(\N)$
  and let $x\in
 \ell^p(\N)$.  
 Fix $l\ge 0$. Suppose that there exist three integers $0\le 
k_{0}<k_{1}<k_2\le b_{l+1}-b_{l}$ such that 
\[
|w_{b_{l}+k}|=1\quad\textrm{for every}\ 
k\in(k_{0},k_{1})\cup (k_2,b_{l+1}-b_{l})\quad\textrm{and}\ 
\prod_{s=b_l+k_{0}+1}^{b_{l+1}-1}|w_{s}|=1.
\]
Then we have for every $J\ge 0$,
\begin{align*}
 &\dfrac{1}{J+1}\ \#\Bigl\{0\le j\le J\,:\, \|P_{l}T^{\,j}P_{l}\,x\|\ge 
\|P_lX\|/2\Bigr\}\\
 &\hspace{3cm}\ge 
1-4\bigl(k_2-k_1+k_0\bigr)\,\Bigl( \frac{1}{J+1}+\frac{1}{b_{
l+1}-b_{l}}\Bigr)\cdot
\end{align*}
\end{prop}

We are now able to state sufficient conditions on parameters $(\eta^{(k)})_k$, $(\delta^{(k)})_k$, $(\tau^{(k)})_k$ and $(\Delta^{(k)})_k$ so that our operator $T$ is not $\mathcal{U}$-frequently hypercyclic.

\begin{prop}\label{Prop2}
If the sequence $(\gamma_{k})_{k\ge 1}$, defined by $\gamma_k:=2^{-\tau 
^{(k)}}\bigl(\Delta ^{(k)} 
\bigr)^{1-\frac1{p}}$ for every $k\ge 1$, is a non-increasing sequence satisfying $\sum_{k\ge 1}2^{k}\gamma_k^{1/2}\le 1$, and if the following conditions hold true:
\[
4\delta^{(k)}+3\eta^{(k)}\le\Delta^{(k)}\le \eta^{(k+1)},\quad \lim\limits_{k\to\infty }\; {\eta^{(k)}}/ {\Delta^{(k)}}=0 \quad\text{and}\quad \lim\limits_{k\to\infty }\; {\delta ^{(k)}}/ {\eta^{(k)}}=0\]
then $T$ is not $\mathcal{U}$-frequently hypercyclic.
\end{prop}

\begin{proof}
It suffices to show that if $x$ is a hypercyclic vector for $T$ then $x$ is not $\mathcal{U}$-frequently hypercyclic for $T$. To this end, we will show that Lemma~\ref{Theorem 48} can be applied to $x$ by considering $\beta_l=4\,\gamma_k$, $N_{l}=\eta^{(k)}$, $R_l=\Delta^{(k)}-2\delta^{(k)}-2\eta^{(k)}$, $L_l=2\delta^{(k)}$ and $C=1/4$ for every 
$l\in[2^{k-1},2^{k})$, every $k\ge 1$.
It will then remain to show that 
\[\liminf_{l\to\infty }\ \inf_{k\ge R_{l}}\ \dfrac{\ \#\bigl\{0\le 
j\le 
k\,:\, \|P_{l}\,T^{\,j}P_{l}\,x\|\ge \|P_l X\|/2\bigr\}}{k+1}=1.\]
We can already remark that $(\beta_l)$ is a non-increasing sequence with $\sum_{l\ge 1}\sqrt{\beta_l}\le 1$, that $(N_l)$ is a non-decreasing sequence tending to infinity, that $(R_l)$ satisfies $N_l\le R_l<\inf_{j\in \bigcup_n \varphi^{-n}(l)}N_j$ and $R_l< b_{l+1}-b_l$, and that $(L_l)$ satisfies $\frac{L_l}{N_l}\to 0$. Moreover, we have:

\begin{enumerate}
 \item[(1)] Since $\prod_{w=i+1}^{b_{l+1}-1} |w_s|\ge 1$ for every $l\ge 0$, every $b_l\le i<b_{l+1}$, we have $\|P_{n}x\|\le \|P_n X\|$ for every $n\ge 0$;
 \item[(2)] We have for every $k\ge 1$ and every $n\in[2^{k-1},2^{k})$,
\[
|v_{n}|\,\cdot\sup_{j\,\in\,[b_{\varphi (n)},b_{\varphi (n)}+1)}
\ \prod_{s=b_{\varphi (n)+1}}^{j}\!\!\!|w_{s}|\le 2
^{-\tau ^{(k)}}.
\]
Therefore, if we set $C_{n}=2
^{-\tau ^{(k)}}$ for every $n\in [2^{k-1},2^{k})$, Proposition~\ref{Proposition 50} implies that 
for every $k\ge 1$, every $l\in[2^{k-1},2^{k})$ and 
every $0\le n<l$,
\[
 \sup_{j\ge 0}\ \|P_{n}T^{\,j}P_{l}\,x\|\le 2^{-\tau 
^{(k)}}\bigl(\Delta ^{(k)} 
\bigr)^{1-\frac1{p}}
\|P_l X\|\le \dfrac{\beta_{l}}{4}
\|P_l X\|;
\]
 \item[(2')] We also deduce from Proposition~\ref{Proposition 50} that 
 \[\|P_{n}\,T^{\,j}P_{l}\,x\|\le 
\dfrac{\beta _{l}}{4} \|P_{[b_{l+1}-j,b_{l+1})}X\|\] for every $l\ge 1$, every $j\le b_{l+1}-b_l$ and every $0\le n<l$;
\item[(3)] Since $w_{i}^{(k)}=1$ if $\Delta^{k}-\eta^{(k)}\le i<\Delta ^{(k)}$, it follows from Proposition~\ref{Proposition 50} and from the definition of $X$ that for every $l\ge 1$ and every $0\le n<l$,
 \[\sup\limits_{0\le j\le N_{l}}\|P_{n}\,T^{\,j}P_{l}\,x\|\le \dfrac{\beta _{l}}{4} \|P_{[b_{l+1}-N_l,b_{l+1})}X\|\le \dfrac{\beta_l}{4}\|P_{l}x\|;\]
\item[(4)] Since $x$ is hypercyclic, we have $\sup\limits_{j\ge 0}\sum\limits_{l>n}\|P_{n}\,T^{\,j}P_{l}x\|=\infty$;
\item[(5)] Let $k\ge 1$, $l\in[2^{k-1},2^{k})$, $\eta^{(k)}\le j\le \Delta^{(k)}-2\delta^{(k)}-2\eta^{(k)}$ and $j+2\delta^{(k)}\le s\le \Delta^{(k)}-2\delta^{(k)}-2\eta^{(k)}$. We need to show that 
\[\|P_l T^{s}P_{[b_{l+1}-j,b_{l+1})}x\|\ge \|P_{[b_{l+1}-j,b_{l+1})}X\|.\]
If $i\in [\Delta^{(k)}-j,\Delta^{(k)})$, we get
\[P_lT^se_{b_l+i}=-\Big(\prod_{\nu=i+1}^{\Delta^{(k)}-1}w^{(k)}_{\nu}\Big)\Big(\prod_{\nu=1}^{i+s-\Delta^{(k)}}w^{(k)}_{\nu}\Big)e_{b_l+i+s-\Delta^{(k)}}.\]
Therefore, since 
\[i+s-\Delta^{(k)}\ge \Delta^{(k)}-j+j+2\delta^{(k)}-\Delta^{(k)}\ge 2\delta^{(k)}\]
and
\[ i+s-\Delta^{(k)}\le \Delta^{(k)}-1+\Delta^{(k)}-2\delta^{(k)}-2\eta^{(k)}-\Delta^{(k)}\le \Delta^{(k)}-2\delta^{(k)}-2\eta^{(k)}-1,\]
we have  $\prod_{\nu=1}^{i+s-\Delta^{(k)}}w^{(k)}_{\nu}=1$ and thus the desired inequality.\\
\end{enumerate}

We deduce from Lemma~\ref{Theorem 48} that there exists $\varepsilon >0$ such that
\[
\underline{\vphantom{p}\textrm{dens}}\ \mathcal{N}_T\bigl(x,B(0,\varepsilon )^c\bigr)\ge 
\liminf_{l\to\infty }\ \inf_{J\ge R_{l}}\ \dfrac{\ \#\bigl\{0\le 
j\le 
J\,:\, \|P_{l}\,T^{\,j}P_{l}\,x\|\ge \|P_l X\|/2\bigr\}}{J+1}.
\]

Since the assumptions of Proposition~\ref{Proposition 51}  are satisfied for $k_0=2\delta^{(k)}$, $k_1=\Delta^{(k)}-2\delta^{(k)}-2\eta^{(k)}$ and $k_2=\Delta^{(k)}-\eta^{(k)}$, we deduce that for every $J\ge 0$, every $k\ge 1$, and every $l\in[2^{k-1},2^{k})$,
\[
\dfrac{1}{J+1}\ \#\Bigl\{0\le j\le J\,:\,\Vert P_{l}T^{\,j}P_{l}\,x\Vert \ge
\|P_lX\|/2\Bigr\}\ge 1-(16\delta^{(k)}+4\eta^{(k)})\Bigl(\dfrac{1}{J+1}+\dfrac{1}
{\Delta ^{(k)}}\Bigr).
\]
Therefore, for every $k\ge 1$ and every $l\in[2^{k-1},2^{k})$, we get
\begin{align*}
&\inf_{J\ge R_{l}}\ 
\dfrac{1}{J+1}\ \#\Bigl\{0\le j\le J\,:\,\Vert P_{l}T^{\,j}P_{l}\,x\Vert \ge
\|P_lX\|/2\Bigr\}\\
&\quad\ge 1-(16\delta^{(k)}+4\eta^{(k)})\Bigl(\dfrac{1}{\Delta^{(k)}-2\delta^{(k)}-2\eta^{(k)}+1}+\dfrac{1}
{\Delta ^{(k)}}\Bigr)
\end{align*}
and since $\lim\limits_{k\to\infty }\frac{\eta^{(k)}}{\Delta 
^{(k)}}=0$ and $\lim\limits_{k\to\infty }\frac{\delta^{(k)}}{\eta^{(k)}}=0$, we conclude that
\[
 \liminf_{l\to\infty }\ \inf_{J\ge R_{l}}\ 
\dfrac{1}{J+1}\ \#\Bigl\{0\le j\le J\,:\,\Vert P_{l}T^{\,j}P_{l}\,x\Vert \ge
\|P_lX\|/2\Bigr\}=1.\] The vector $x$ is thus not $\mathcal{U}$-frequently hypercyclic since $\underline{\vphantom{p}\textrm{dens}}\ \mathcal{N}_T\bigl(x,B(0,\varepsilon )^c\bigr)=1$ and thus $\overline{\vphantom{p}\textrm{dens}}\ \mathcal{N}_T\bigl(x,B(0,\varepsilon )\bigr)=0$.
\end{proof}

\begin{theorem}
Let $1\le p<\infty$. There exists an invertible operator $T$ on $\ell^p(\mathbb{N})$ such that $T$ is $\mathcal{U}$-frequently hypercyclic and $T^{-1}$ is not $\mathcal{U}$-frequently hypercyclic.
\end{theorem}
\begin{proof}
In view of Propositions ~\ref{Prop1} and \ref{Prop2}, it suffices to show that there exists increasing sequences of parameters $(\tau^{(k)})$, $(\delta^{(k)})$, $(\eta^{(k)})$ and $(\Delta^{(k)})$ satisfying 
\begin{itemize}
\item $\tau^{(k)}\ge 2^k$ for every $k\ge 1$ (see \eqref{tau});
\item $\lim_k \delta^{(k)}-\tau^{(k)}=\infty$;
\item  the sequence $(\gamma_{k})_{k\ge 1}$, defined by $\gamma_k:=2^{-\tau 
^{(k)}}\bigl(\Delta ^{(k)} 
\bigr)^{1-\frac1{p}}$, is a non-increasing sequence satisfying $\sum_{k\ge 1}2^{k}\gamma_k^{1/2}\le 1$;
 \item $4\delta^{(k)}+3\eta^{(k)}\le\Delta^{(k)}\le \eta^{(k+1)}$ for every $k\ge 1$;
\item $\lim\limits_{k\to\infty }\; {\eta^{(k)}}/ {\Delta^{(k)}}=0$;
\item  $\lim\limits_{k\to\infty }\; {\delta ^{(k)}}/ {\eta^{(k)}}=0$.
\end{itemize}

We can compute that the following choice of parameters satisfies each of these conditions if the constant $C$ is chosen sufficiently large:
\[\tau^{(k)}=2^{Ck^2} ,\quad \delta^{(k)}=2^{Ck^2+1},\quad \eta^{(k)}=2^{2Ck^2+1} \quad \text{and}\quad \Delta^{(k)}=2^{2Ck^2+k+4}.\]

\end{proof}

\end{document}